\documentclass[10pt]{amsart}
\usepackage{latexsym, amsmath, amsfonts, amssymb, mathrsfs, fancyhdr, tikz, tikz-cd}
\usepackage{verbatim}
\usetikzlibrary{matrix,arrows,decorations.pathmorphing,calc}

\def\ds{\displaystyle}
\def\nin{\not \in}
\def\n{\noindent}
\def\R{\mathbb{R}}
\def\Rpq{\mathbb{R}^{p,q}}
\def\eij{e_{ij}}
\def\v{\varphi}
\def\e{\varepsilon}
\def\s{\sigma}

\def\e{\epsilon}
\def\d{\delta}

\def\T{\mathcal{T}}
\def\l{\lambda}

\def\X{\mathcal{X}}
\def\V{\mathcal{V}}
\def\S{\mathcal{S}}
\def\E{\mathcal{E}}
\def\a{\alpha}
\def\b{\beta}
\def\Bx{\mathcal{B}_x}
\def\g{\gamma}

\newtheorem{theorem}{Theorem}
\newtheorem{lemma}[theorem]{Lemma}
\newtheorem{cor}[theorem]{Corollary}

\theoremstyle{definition}
\newtheorem{definition}[theorem]{Definition}

\newtheorem{xca}[theorem]{Exercise}
\newtheorem{remark}{Remark}

\theoremstyle{remark}

\numberwithin{equation}{section}

\begin{document}

\title{Simplicial isometric embeddings of indefinite metric polyhedra}

\author{Barry Minemyer}
\address{Department of Mathematics, The Ohio State University, Columbus, Ohio 43210}
\email{minemyer.1@osu.edu}



\date{\today.}



\begin{abstract}

In this paper, isometric embedding results of Greene, Gromov and Rokhlin are extended to what are called ``indefinite metric polyhedra".  An {\it indefinite metric polyhedron} is a locally finite simplicial complex where each simplex is endowed with a quadratic form (which, in general, is not necessarily positive-definite, or even non-degenerate).  It is shown that every indefinite metric polyhedron (with the maximal degree of every vertex bounded above) admits a simplicial isometric embedding into Minkowski space of an appropriate signature.  A simple example is given to show that the dimension bounds in the compact case are sharp, and that the assumption on the upper bound of the degrees of vertices cannot be removed.  
These conditions can be removed though if one allows for isometric embeddings which are merely piecewise linear instead of simplicial.

\end{abstract}

\maketitle



\section{Introduction}\label{section 1}

Since before the famous works of Euler, Gauss and Riemann, mathematicians have been interested in the study of manifolds.  This led to the natural question of whether or not we could ``realize" manifolds in some appropriate sense.  Answering these question(s) are the celebrated embedding theorems of Whitney (\cite{Whitney1} and \cite{Whitney2}) and the isometric embedding theorems for Riemannian manifolds by Nash (\cite{Nash1} and \cite{Nash2}).  In 1970 Robert Greene \cite{Greene} and M. L. Gromov and V. A. Rokhlin \cite{Gromov1} independently proved that every manifold endowed with an indefinite metric tensor admits an isometric embedding into Minkowski space of an appropriate signature, thus extending Nash's results to the non-Riemannian case.  

While the concept of a metric space also goes back a long ways, the area of mathematics now referred to as ``metric geometry" has really only exploded over the past 60 some years.  A major theme in current research is to attempt to extend results from differential geometry to various types of metric spaces.  While it is a well known fact that every metric space with Lebesgue covering dimension $n$ admits an embedding into $\mathbb{R}^{2n+1}$, any attempt to generalize the above isometric embedding results of manifolds is going to be difficult, however, because they will depend greatly on the properties of that metric space.  For example, it is not true that even some (relatively) well-behaved metric spaces admit isometric embeddings into Euclidean space.  For example, as noted by Petrunin in \cite{Petrunin1} and Le Donne in \cite{Donne}, any Finsler manifold which is not Riemannian does not admit such an isometric embedding.  

There is still hope, however, of extending the results of Nash, Greene, Gromov and Rokhlin to various collections of spaces.  In \cite{Gromov2}, Gromov posed the question of whether or not Euclidean polyhedra admit piecewise linear (pl) isometric embeddings into Euclidean space, where a \emph{Euclidean polyhedron} is a metric space $\X$ which admits a locally finite triangulation $\T$ such that every $k$-dimensional simplex $S \in \T$ is affinely isometric to a $k$-dimensional simplex in $\mathbb{E}^k$ for all $k$.  Note that pl maps are the best that we can hope for here as, in general, we have no shot at finding simplicial isometric embeddings.  To see this simply consider the 1-skeleton of the standard 2-simplex (so, a triangle) where two of the edges have length 1 and the third edge has length 100.  It is easily seen that this space does not admit a simplicial isometric embedding into $\mathbb{E}^N$ for any $N$ but admits a pl isometric embedding into $\mathbb{E}^2$ (just subdivide the long edge).

Gromov's question above has, for the most part\footnote{That is, this problem has been solved up to either minimizing the dimensionality of the target Euclidean space or showing that the current bound is minimal.}, been answered in the affirmitive.  Some of the mathematicians associated with this are Zalgaller \cite{Zalgaller}, Krat \cite{Krat}, Akopyan \cite{Akopyan}, and the author \cite{Minemyer1}, \cite{Minemyer2}.  The present article extends the indefinite metric results of Greene, Gromov and Rokhlin to what are called \emph{indefinite metric polyhedra}.  

The first question is ``what should be the correct indefinite metric analogue to Euclidean polyhedra"?  In a Euclidean polyhedron, each simplex is affinely isometric to a simplex in Euclidean space.  This allows one to associate with each $k$-dimensional simplex a positive definite quadratic form on $\R^k$.  Furthermore, these forms ``agree" on the intersection of adjacent simplices, meaning that their restrictions yield the same length function on this intersection.  So the correct definition for an indefinite metric polyhedron should be the same as the above, except that the quadratic form associated with each simplex need not be positive definite (or even non-degenerate).  In section \ref{section 7} this definition is made precise and is shown to be equivalent to simply assigning real numbers to each edge of the underlying simplicial complex.  More specifically, an \emph{indefinite metric polyhedron} is defined to be a triple $(\X, \T, g)$ where $\X$ is a topological space, $\T$ is a triangulation of $\X$ with edge set $\E$, and $g: \E \rightarrow \R$ is simply a function.  This latter definition is much more convenient for proving \emph{simplicial} isometric embedding results, which will now be discussed.

What is most interesting in considering maps into Minkowski space is that one may return to the possibility of simplicial isometric embeddings.  For instance, the earlier example \emph{does} admit a simplicial isometric embedding into $\R^1_1$!  One such embedding is a map sending the vertex opposite the long edge to $(0,0)$ and sending the other two vertices to $(\pm 50, 7 \sqrt{51})$.  In general, \emph{every} compact indefinite metric polyhedron admits a simplicial isometric embedding into Minkowski space (of an appropriate signature), and for the non-compact case this is true assuming that the maximal degree of every vertex is bounded above\footnote{This condition is necessary, as we will see in section \ref{section 6}.}.

The following three theorems are proved in this paper.

\vskip 10pt

\begin{theorem}\label{main thm 1}

Let $(\X, \T, g)$ be a compact $n$-dimensional indefinite metric polyhedron with vertex set $\V$.  Let $d = \text{max} \{ \text{deg}(v) | v \in \V \} $ and let $q = \text{max} \{ d, \, 2n + 1 \}$.  Then there exists a simplicial isometric embedding of $\X$ into $\ds{\R^{q,q}}$. 

\end{theorem}

\vskip 10pt

\begin{theorem}\label{main thm 2}

Let $(\X, \T, g)$ be a compact $n$-dimensional indefinite metric polyhedron with edge set $\E$.  Then $\X$ admits a simplicial isometric embedding into $\R^{p,q}$ for some integers $p$ and $q$ which satisfy $p \geq 2n + 1$ and\footnote{Actually, we can divide these $2n + 1$ coordinates between $p$ and $q$ in any way that we wish.} $p + q = 2n + 1 + |\E|$.

\end{theorem}

\vskip 10pt

\begin{theorem}\label{main thm 3}

Let $(\X, \T, g)$ be an $n$-dimensional indefinite metric polyhedron with vertex set $\V$ and suppose that $d = \text{max} \{ \text{deg}(v) | v \in \V \} < \infty$. Let $\ds{q = \text{max} \{ d, \, 2n + 1 \} }$.  Then there exists a simplicial isometric embedding of $\X$ into $\ds{\R^{p,p}}$ where $p = 2q(d^3 - d^2 + d + 1)$.

\end{theorem}

\vskip 10pt

Notice that Theorems \ref{main thm 1} and \ref{main thm 2} are essentially the same but, due to the existence of the $|\E|$ term in Theorem \ref{main thm 2}, the dimension requirements in Theorem \ref{main thm 1} will generally be much smaller.  The reason Theorem \ref{main thm 2} is included is because the proof is somewhat constructive while the proof of Theorem \ref{main thm 1} is completely existential.  It is certainly possible that the method of proof of Theorem \ref{main thm 2} will end up as the most important result in this paper from an applied viewpoint.  

An outline of the paper is as follows.  In section \ref{section 2} a few preliminary definitions and facts will be discussed.  Theorem \ref{main thm 1} will be proved in section \ref{section 3} and Theorem \ref{main thm 2} in section \ref{section 4}.  The proof of Theorem \ref{main thm 3} uses Theorem \ref{main thm 1} and will be completed in section \ref{section 5}.  Then in section \ref{section 7} it is shown that the two definitions of an indefinite metric polyhedron coincide.  Finally, in section \ref{section 6} an example is presented which shows that both the dimension requirements in Theorem \ref{main thm 1} are sharp and that the assumption ``$d < \infty$" in Theorem \ref{main thm 3} is necessary.  

For most\footnote{The dimension requirements for the $C^1$ Nash-Kuiper results are certainly sharp.} isometric embedding results it is not known whether or not the dimension(s) of the target space are optimal, which is one very nice feature of Theorem \ref{main thm 1}.  Note, though, that there is certainly no mention of whether or not the dimension bounds for Theorem \ref{main thm 3} are sharp as it is believed that they are not\footnote{It is now known that they are not.  See Remark \ref{remark: new paper} below.}.  

One final remark about the above results.
The amount of necessary codimension in the above three theorems is quite high, and in that way these simplicial isometric embeddings resemble the rigidity of $C^k$-isometric embeddings ($k \geq 2$) of Riemannian manifolds into Euclidean space.
One may wonder if this rigidity necessity is preserved in the more general setting of {\it piecewise linear} isometric embeddings.
But the following Theorem, which is a specific case of a more general result proved in \cite{Minemyer3}, shows that this is not the case.

\vskip 10pt

\begin{theorem}\label{main thm 4}

Let $(\X, \T, g)$ be an n-dimensional indefinite metric polyhedron (with a locally finite triangulation $\T$).  Then $\X$ admits a piecewise linear isometric embedding into $\R^{2n,n}$.

\end{theorem}

\vskip 10pt

So by Theorem \ref{main thm 4} and the example in section \ref{section 6} it is seen that the dimension requirements for simplicial isometric embeddings depend on the maximal degree $d$ of any vertex of $\T$, but when the setting is passed to pl isometric embeddings this requirement is removed! 
In particular, the collection of piecewise linear isometric embeddings into Minkowski space seem to satisfy some sort of discrete {\it $h$-principle}, while their simplicial counterparts do not.

\begin{remark}\label{remark: new paper}
The main results of this paper were originally posted on arXiv in November of 2012, and were first submitted for publication during the summer of 2013.  
They are also contained in the author's Ph. D. Thesis \cite{Minemyer2}.
In January of 2015, a paper was posted on arXiv by Pavel Galashin and Vladimir Zolotov \cite{GZ} which extends some of these results.  
Specifically, they show that the dimension requirements of Theorem \ref{main thm 3} can be reduced to the same as those of Theorem \ref{main thm 1}.  
Their method is also pretty algorithmic, so in that sense it is an improvement on Theorem \ref{main thm 2} and Section \ref{section 4} below.
\end{remark}

\subsection*{Acknowledgements} The motivation for many of the ideas in this paper comes from Nash \cite{Nash2} and Greene \cite{Greene}.  The author also wants to thank P. Ontaneda and many others for helpful remarks and guidance during the writing of this article.  This research was partially supported by the NSF grant of Tom Farrell and Pedro Ontaneda, DMS-1103335.


\section{Preliminaries}\label{section 2}

\subsection{General position}

A set of $k$ points in $\R^N$ (with $k \leq N + 1$) is said to be \emph{affinely independent} if the entire set of points is not contained in any $(k - 2)$-dimensional affine subspace of $\R^N$.  A set of points $\mathscr{A}$ in $\mathbb{R}^N$ is said to be in \emph{general position} if every subset of $\mathscr{A}$ containing $N + 1$ or fewer points is affinely independent.  Suppose $n$ and $N$ are integers with $n \leq N$.  A set of points $\mathscr{B}$ in $\mathbb{R}^N$ is said to be in \emph{$n$-general position} if every subset of $\mathscr{B}$ containing $n + 1$ or fewer points is affinely independent.

For this paper, a \emph{polyhedron} is a tuple $(\X, \T)$ where $\X$ is a topological space and $\T$ is a locally finite simplicial triangulation of $\X$.  \emph{Simp($\X, \R^N$)} denotes the collection of all simplicial maps from $\X$ into $\R^N$ (with respect to $\T$) and \emph{Met($\X$)} denotes the collection of all indefinite metrics on $\X$ as defined in section \ref{section 1}.

Notice that if $(\X, \T)$ is a compact polyhedron and if we fix an ordering on the vertex set $\V$ and the edge set $\E$ of $\T$ that we have a bijective correspondence between Simp($\X, \R^N$) $\cong$ $\R^{N|\V|}$ and Met($\X$) $\cong$ $\R^{|\E|}$.  This allows us to consider both Simp($\X, \R^N$) and Met($\X$) as topological vector spaces.  And this remark does not change if we replace the Euclidean inner product on $\R^N$ with any Minkowski inner product (see subsection 2.2).  

The following well known Lemma is proved in \cite{HY}:

\begin{lemma}\label{lemma:important}

Let $(\X, \T)$ be an $n$-dimensional polyhedron, let $f$ $\in$ $\text{Simp}(\X, \mathbb{R}^N)$, and let $\V$ be the vertex set of $\X$.  Let $f(\V)$ denote the collection of the images of the vertices of $\T$.  If $f(\V)$ is in ($2n + 1$)-general position (so in particular we must have $N$ $\geq$ $2n + 1$), then $f$ is an embedding.

\end{lemma}

\begin{cor}\label{corollary:important}

Let $(\X, \T)$ be an $n$-dimensional polyhedron, let $f$ $\in$ $\text{Simp}(\X, \mathbb{R}^N)$, and let $\V$ be the vertex set of $\T$.  Let $f(\V)$ denote the collection of the images of the vertices of $\T$.  If $f(\V)$ is in ($2n$)-general position (so in particular we must have $N$ $\geq$ $2n$) then $f|_{St(v)}$ is an embedding for any vertex $v$, where $St(v)$ denotes the closed star of the vertex $v$.  If $f$ is in $n$-general position (so $N \geq n$) then $f$ is an immersion.

\end{cor}

In Corollary \ref{corollary:important} an \emph{immersion} is a map $f \in \text{Simp}(\X, \mathbb{R}^N)$ such that the restriction $f|_{\s}$ is an embedding for each simplex $\s \in \T$.  This mimics the definition from differential geometry as such a map is injective on the tangent space at each point\footnote{Where the tangent space at each point makes sense since each point of $\X$ is interior to a unique simplex of $\T$.  But, of course, the tangent spaces at different points may have different dimensions, and in particular the tangent space at each vertex is 0-dimensional.}.  But note that, when considering polyhedra, not all immersions are locally injective (for a simple example, please see \cite{Minemyer2}).

\subsection{Minkowski space $\Rpq$}

\emph{Minkowski space of signature $(p, q)$}, denoted by $\R^{p,q}$, is $\R^{p + q}$ endowed with the symmetric bilinear form of signature $(p,q)$.  More specifically, if $\vec{v}, \vec{w} \in \R^{p,q}$ with $\vec{v} = (v_i)_{i = 1}^{p + q}$ and $\vec{w} = (w_i)_{i = 1}^{p + q}$, then
	\begin{equation*}
	\langle \vec{v} , \vec{w} \rangle_{\R^{p,q}} := \langle \vec{v} , \vec{w} \rangle := \sum_{i = 1}^p v_i w_i - \sum_{j = p + 1}^{p + q} v_j w_j .
	\end{equation*}
There will be less concern in later parts of this paper with respect to the first $p$ coordinates of $\Rpq$ being the ``positive" coordinates with respect to $\langle,\rangle$, and in general the inner product will be written in the form 
	\begin{equation*}
	\langle \vec{v}, \vec{w} \rangle \, = \sum_{i = 1}^{p + q} \sigma(i) v_i w_i 
	\end{equation*}
where $\sigma(i) = 1$ for $p$ (fixed) coordinates and $\sigma(i) = -1$ for the other $q$ coordinates.
			
The use of $\Rpq$ will specifically mean $\R^{p + q}$ endowed with the symmetric bilinear form of signature $(p,q)$, $\mathbb{E}^N$ will mean $\R^N$ with the symmetric bilinear form of signature $(N,0)$, and $\R^N$ will mean to include the possibility of \emph{any} Minkowski inner product of signature $(p',q')$ such that $p' + q' = N$.
			
Define the \emph{signed square} function $s:\R \rightarrow \R$ by 
	\begin{equation*}
	\ds{s(x) = \left\{ \begin{array}{rr} x^2 & \text{   if } x \geq 0 \\ -x^2 & \text{   if } x < 0 \end{array} \right.  }
	\end{equation*}
If $g \in \text{Met}(\X)$, then define $g^2 \in \text{Met}(\X)$ by $g^2(e) := s(g(e))$ for any edge $e$ of $\T$.  Then a simplicial isometric embedding of $(\X, \T, g)$ into $\Rpq$ is defined to be an embedding $h \in \text{Simp}(\X, \Rpq)$ which satisfies that for any edge $\eij \in \E$ between the vertices $v_i$ and $v_j$: 
	\begin{equation*}
	\langle (h(v_i) - h(v_j)) , (h(v_i) - h(v_j)) \rangle = g^2(\eij).
	\end{equation*}
			
This definition is analogous to that of an affine isometric embedding of a simplex into Euclidean space.  
For example, if the edge $e_{ij}$ between two vertices $v_i$ and $v_j$ has intrinsic length 3, 
then we want our isometry $h$ to satisfy that $\langle h(v_i) - h(v_j)) , (h(v_i) - h(v_j)) \rangle = 9$.  
So in exactly the same way, if the intrinsic ``length" of $e_{ij}$ is -3 then we want $\langle h(v_i) - h(v_j)) , (h(v_i) - h(v_j)) \rangle = -9$.


\section{Proof of Theorem \ref{main thm 1}}\label{section 3}

For the remainder of sections \ref{section 3} and \ref{section 4}, $(\X, \T, g)$ will denote a \emph{compact} indefinite metric polyhedron.  So $\T$ is assumed to be finite. 

\subsection{The map $\v$}  

Define 
	\begin{equation*}
	\varphi : \text{Simp}(\X, \R^N) \rightarrow \text{Met}(\X) 
	\end{equation*}
to be the square of the induced metric map.  That is, if $ f \in \text{Simp}(\X, \R^N) $ and $\eij \in \E$, define  
	\begin{align*} 
	\varphi(f)(\eij) &= \langle (f(v_i) - f(v_j)) , (f(v_i) - f(v_j)) \rangle \\ 
	&= \sum_{k = 1}^{N} \sigma(k) (f_k (v_i) - f_k(v_j))^2
	\end{align*}
where $\eij$ is the edge between the vertices $v_i$ and $v_j$, and $(f_k)_{k = 1}^{N} $ are the component functions of $f$.

As defined above, the domain of the map $\v$ technically depends on $N$.  But by an abuse of notation this will not be considered.  So, for example, one can talk about 2 metrics $\v(h_1)$ and $\v(h_2)$ where $h_1 \in \text{Simp}(\X, \R^{N_1})$ and $h_2 \in \text{Simp}(\X, \R^{N_2})$, but $N_1 \neq N_2$.  In discussing the square of an induced metric $\v(h)$, it will always be clear into what dimensional space that $h$ is defined.  The map $\v$ also depends on which Minkowski inner product is being considered on $\R^N$.  But in the analysis below it will be possible to consider all of the different inner products at once.

The reason to consider the {\it square} of the induced metric map is because it is, in some sense, ``linear over addition in $\text{Met}(\X)$".  To make this precise, let $ \alpha \in \text{Simp}(\X, \R^{N_1}) $ and let $ \beta \in \text{Simp}(\X, \R^{N_2})$.  Since $\text{Met}(\X)$ is a vector space one can consider $\v(\alpha) + \v(\beta)$.  Then:
	\begin{align*}
	(\v(\alpha) + \v(\beta))(\eij) &:= \v(\alpha)(\eij) + \v(\beta)(\eij)  \\
	&=  \sum_{k = 1}^{N_1} \sigma(k) (\alpha_k (v_i) - \alpha_k (v_j))^2 + \sum_{k = 1}^{N_2} \sigma(k) (\beta_k (v_i) - \beta_k (v_j))^2  \\
	&= \v(\alpha \oplus \beta)(e_{ij}) 
	\end{align*}
where $ \alpha \oplus \beta \in \text{Simp}(\X, \R^{N_1 + N_2}) $ is the concatenation of the maps $\alpha$ and $\beta$, and where the Minkowski inner product on $\R^{N_1 + N_2}$ is determined by the inner products given to $\R^{N_1}$ and $\R^{N_2}$.    
 
One easy property of the map $\v$, but one which will be used later, is $ \v(\lambda f) = \lambda^2 \v(f)$ for all $ \lambda \in \R$ and for all $ f \in \text{Simp}(\X, \R^N)$.  To see this, just note that 
	\begin{equation*}
	\v(\lambda f)(\eij) = \sum_{k = 1}^{N} \sigma(k) (\lambda f_k(v_i) - \lambda f_k(v_j))^2  = \lambda^2 \sum_{k = 1}^{N} \sigma(k) (f_k(v_i) - f_k(v_j))^2 = \lambda^2 \v(f)(\eij).
	\end{equation*}

\subsection{The differential of $\v$}  The next Lemma is the key to prove Theorem \ref{main thm 1}.

\begin{lemma}\label{lemma surjective differential}

Let $d = \text{max} \{ \text{deg}(v) | v \in \V \}$ and let $f \in Simp(\X, \R^N)$ with $N \geq d$.  If the images of the vertices of $\T$ under $f$ are in $d$-general position, then the differential of $\v$ at $f$ has rank $|\E|$ (and where $\E$ is the edge set of $\T$).

\end{lemma}

\begin{proof}

If $N \in \mathbb{N}$ is fixed, then the Jacobian Matrix of $\v$ will be an $|\E| \times N|\V|$ matrix.  So as a first observation note that for $d\v$ to be surjective at any point we must have $N|\V| \geq |\E| \, \Longrightarrow \, N \geq \frac{|\E|}{|\V|}$.  Let $ f \in \text{Simp}(\X, \R^N) $ with component functions $(f_k)_{k = 1}^{N}$ and let $\eij$ denote the edge of $\T$ connecting the vertices $v_i$ and $v_j$.  Let $\v_{\eij}$ denote the $\eij$ component of the map $\v$ (thought of as a map from $\R^{N |\V|}$ to $\R^{|\E|}$) and let $ f^{i}_{k} := f_k(v_i)$ for all $1 \leq i \leq |V|$ and for all  $1 \leq k \leq N $.  Note that in this notation $\v_{\eij}(f) = \sum_{k = 1}^{N} \s(k) (f^{i}_{k} - f^{j}_{k})^2$.  Then we compute:  
	\begin{equation*}
	\frac{\partial \v_{\eij}}{\partial f^{l}_{k}} = \left \{
                                                                    \begin{array}{ll}
                                                                                        0                               &\text{if } l \neq i, j \\
                                                                                        2\s(k) (f^{i}_{k} - f^{j}_{k})        &\text{if } l = i \\
                                                                                        2\s(k) (f^{j}_{k} - f^{i}_{k})        &\text{if } l = j \\
                                                                    \end{array}
                                                                        \right. 
 	\end{equation*}
                                                                        
To prove Lemma \ref{lemma surjective differential} what is needed is for the rows of $d\v|_{f}$ to be linearly independent when considered as vectors in $\R^{N|\V|}$.  Since multiplying a column of a matrix by a non-zero constant does not change the rank of the matrix, each column of $d\v|_{f}$ can be multiplied by either $\frac{1}{2}$ or $\frac{-1}{2}$ in order to remove the $2 \s(k)$.  What follows is, for an arbitrary edge $\eij$ of $\E$, an analysis of the row of $d\v|_{f}$ corresponding to this edge.

The matrix $d\v|_{f}$ has $N|\V|$ columns.  But it is easier to see what is happening if one considers $d\v|_{f}$ as having $|\V|$ columns, the entries of which are row vectors of $\R^N$.  These columns will be called \emph{block columns} of $d\v|_{f}$ (so in particular $d\v|_{f}$ has $|\V|$ block columns).  Using this notation, it can be seen that the row of $d\v|_{f}$ corresponding to the edge $\eij$ looks like: 
	\begin{equation*}
	\left[ \vec{0} \right| \hdots \left| \vec{0} \right| f(v_i) - f(v_j) \left| \vec{0} \right| \hdots \left| \vec{0} \right| f(v_j) - f(v_i) \left| \vec{0} \right| \hdots \left| \vec{0} \right]  
	\end{equation*}
where the vertical lines are intended to break up the row into $|V|$ block columns.  The $f(v_i) - f(v_j)$ occurs in the $i^{th}$ block column and similarly $f(v_j) - f(v_i)$ is in the $j^{th}$ block column.  Notice that if $f(v_i) = f(v_j)$, then this is the 0 row and $d\v|_{f}$ is therefore not surjective.  So another necessary condition for $d\v|_{f}$ to be surjective is that $f(v_i) \neq f(v_j)$ for all adjacent vertices $v_i, v_j \in \V$.  Next, notice that the $i^{th}$ entry of the row corresponding to the edge $\eij$ is $f(v_j) - f(v_i) := f(\eij)$.  This is just the vector in $\R^N$ whose initial point is $f(v_i)$ and whose terminal point is $f(v_j)$.

Now to see whether or not $d\v|_{f}$ is surjective, consider the block column corresponding to the vertex $v_i$.  The non-zero entries of this column correspond exactly to the edges of $f(\X)$ (considered as vectors in $\R^N $) that are incident with the vertex $f(v_i)$.  So if the set of edges of $f(\X)$ incident with $f(v_i)$ is linearly independent, then the block column corresponding to the vertex $v_i$ will have maximal rank (when considered as an $|\E| \times N$ matrix).  Then let $\ds{ d = \text{max} \{ \text{deg}(v) | v \in \V \} }$ and suppose that $N \geq d$.  (Note in particular that if $N \geq d$, then $N|\V| \geq |\E|$ as required above.)  Then for the block column of $d\v|_{f}$ corresponding to the vertex $v_i$, the rank of the block column will be greater than or equal to $\text{min} \{ \text{deg}(v_i), \, N \} = \text{deg}(v_i)$.  So if the set of edges of $f(\X)$ incident with $f(v_i)$ is linearly independent, then the rank of the block column of $d\v|_{f}$ corresponding to the vertex $v_i$ will equal $\text{deg}(v_i)$.  Or, in other words, the rows of $d\v|_{f}$ corresponding to edges of $\X$ which are incident with $v_i$ are linearly independent.  Thus, if the set of edges of $f(\X)$ at \emph{every} vertex is linearly independent, then $d\v|_{f}$ will have rank equal to $|\E|$ and will therefore be surjective.  This criteria is met if the images of the vertices of $\T$ under $f$ are in $d$-general position.

\end{proof}

Lemma \ref{lemma surjective differential} motivates the following definition:

\begin{definition}

An embedding of $\X$ into $\R^N$ whose vertices are in $d$-general position is called a \emph{free embedding}.

\end{definition}

On a historical note, John Nash in \cite{Nash2} and M.L. Gromov with V.A. Rokhlin in \cite{Gromov1} study embeddings of manifolds where the Nash inverse function theorem applies.  Nash called these embeddings \emph{perturbable} because these were exactly the maps whose images he could {\it perturb} to induce the metric change that he wanted.  But later Gromov and Rokhlin called these embeddings \emph{free} because that more closely described the property that the embedding had to satisfy.  And that property in the case of manifolds was that the collection of first and second order partial derivatives of the embedding function be linearly independent at every point, which is strikingly similar to the property that is needed above in the case of embeddings of polyhedra.  So the same terminology as \cite{Gromov1} is used in order to be consistent.

An easy observation is the following:

\begin{lemma}\label{lemma almost all are free}

Let $(\X, \T)$ be a compact $n$-dimensional polyhedron with vertex set $\V$ and let $ d = \text{max} \{ \text{deg}(v) | v \in \V \} $.  Let $N \geq \text{max} \{d, 2n + 1 \}$ and endow $\ds{\text{Simp}(\X, \R^N)}$ with the canonical Lebesgue measure from $\R^{N|\V|}$.  Then the collection of maps which are \textbf{not} free embeddings has measure 0.  Thus, $d\v|_{f}$ is surjective for almost all $f$ in $\ds{\text{Simp}(\X, \R^N)}$.

\end{lemma}

\vskip 10pt

\subsection{Proof of Theorem \ref{main thm 1}}

The main idea in the following proof is a trick due to Greene in \cite{Greene}

\begin{proof}[Proof of Theorem \ref{main thm 1}]

Let $f$ be a free simplicial embedding of $\X$ into $\mathbb{E}^q$, the existence of which is guaranteed by Lemma \ref{lemma almost all are free}.  Then the map $f \oplus f \, : \X \rightarrow \R^{q,q}$ induces the $\vec{0}$ metric in $\R^{q,q}$, that is $\v(f) = \vec{0}$.  $f \oplus f$ is free since $f$ is.  So by the Inverse Function Theorem, there exists a neighborhood $U$ of $f \oplus f$ in $\text{Simp}(\X, \R^{q,q}$ and a neighborhood $V$ of $\vec{0}$ in $\text{Met}(X)$ such that $\v$ maps $U$ onto $V$.  Note that since $f \oplus f$ is an embedding, $U$ can be chosen to be a open neighborhood in the set of embeddings of $\X$ into $\R^{q,q}$. 

 Now, choose $\lambda > 0$ large enough so that $\frac{g^2}{\lambda^2} \in V$.  Then there exists an embedding $h \in U $ such that $\v(h) = \frac{g^2}{\lambda^2}$.  So $\lambda^2 \v(h) = g^2$ and thus $\v(\lambda h) = g^2$.  Therefore, $\lambda h$ is an isometric embedding of $\X$ into $\R^{q,q}$.

\end{proof}

The following Corollary is really a Corollary to the proof of Theorem \ref{main thm 1} and simply states the required dimensions for isometric local embeddings and for isometric immersions.
	
\vskip 10pt
	
\begin{cor}
Let $(\X, \T, g)$ be a compact $n$-dimensional indefinite metric polyhedron with vertex set $\V$.  Let $d = \text{max} \{ \text{deg}(v) | v \in \V \} $ and let $q^\prime = \text{max} \{ d, \, 2n \}$.  Then there exists a simplicial isometric local embedding of $\X$ into $\ds{\R^{q^\prime, q^\prime}}$ and there exists an isometric immersion of $\X$ into $\R^{d,d}$. 
\end{cor}

\vskip 10pt

\section{Proof of Theorem \ref{main thm 2}}\label{section 4}

\subsection{Simplicial maps with spanning metrics}

Let $f \in \text{Simp}(\X, \R^N)$ and let $ (f_k)^{l}_{k = 1} $ be the component functions of $f$ where for all $1 \leq k \leq l,$  $f_k \in \text{Simp}(\X, \R^{N_k})$.  Note that $\v(f_k) \in \text{Met}(\X)$ for all $k$, and $N = \sum_{k = 1}^{l} N_k$.  $f$ is defined to have a \emph{spanning metric} if the collection $ \{ \v(f_k) \}^{l}_{k = 1}$ spans $\text{Met}(\X)$.

\begin{lemma}\label{lemma 4.1}
Let $m < |\E|$ and suppose that the set $\ds{ \mathbb{A} = \{ g_1, g_2, ... , g_m \} \subset \text{Met}(\X) }$ is linearly independent.  Then the set
\[
\ds{ \mathbb{B} = \{ h \in \text{Simp}(\X, \R^d) | \mathbb{A} \cup \{ \v(h) \} \text{ is linearly independent} \} }
\]
 is dense in $\text{Simp}(\X, \R^d)$, where $\ds{ d = \text{max} \{ \text{deg}(v) | v \in \V \} }$.
\end{lemma}

\begin{proof}
Let $f \in \text{Simp}(\X, \R^d)$ and let $\e > 0$.  What is needed is to construct $h \in \mathbb{B}$ such that $|f - h| < \e$, where $|f - h|$ denotes the Euclidean metric on $\text{Simp}(\X, \R^d) \cong \R^{d|\V|}$.  By Lemma \ref{lemma almost all are free}, almost all $f' \in \text{Simp}(\X, \R^d)$ are free.  So choose $f' \in \text{Simp}(\X, \R^d)$ free such that $|f - f'| < \frac{\e}{2}$.  Now consider $f'$.  If $f' \in \mathbb{B}$, then we are done.  So suppose that $f' \nin \mathbb{B}$, which in particular means that $\v(f') \in \text{Span}(\mathbb{A})$.  Since $f'$ is free, there exists neighborhoods $U$ of $f'$ and $V$ of $\v(f')$ such that $\v$ maps $U$ onto $V$.  By intersecting $U$ with the sphere of radius $\frac{\e}{2}$ centered at $f'$, we may assume that $U$ is contained in the sphere of radius $\frac{\e}{2}$ centered at $f'$.  Then since $\text{Span}(\mathbb{A})$ is contained in a $|\E| - 1$ dimensional subspace of $\text{Met}(\X)$, it has measure 0 in $\text{Met}(\X)$ and therefore almost all points of $V$ do not lie in $\text{Span}(\mathbb{A})$.  So choose $\alpha \in V \setminus \text{Span}(\mathbb{A})$.  Then by the Inverse Function Theorem, there exists $h \in U$ such that $\v(h) = \alpha$.  So $h \in \mathbb{B}$ and $|f - h| \leq |f - f'| + |f' - h| < \frac{\e}{2} + \frac{\e}{2} = \e$.
\end{proof}

\begin{cor}\label{corollary 4.1}
There exists a simplicial map with a spanning metric in $\text{Simp}(\X, \R^{|\E|})$.
\end{cor}

\n {\bf Remark: }
Notice that if $f \in \text{Simp}(\X, \R^{|\E|})$ with component functions $\{ f_k \}_{k = 1}^{|\E|}$ has a spanning metric, then the collection $(\v(f_k))_{k = 1}^{|\E|}$ is a basis for Met($\X$).

\begin{proof}[Proof of Corollary \ref{corollary 4.1}]
The component functions of the simplicial map with a spanning metric $f = \{ f_k \}_{k = 1}^{|\E|}$ will be defined recursively.  Define $f_1: \X \rightarrow \R$ to be any simplicial map which does not map all of the vertices of $\T$ to the same point (and thus does not induce the 0 metric).  So $\v(f_1) \neq \vec{0}$ in Met($\X$).  

Now suppose $f_1, ..., f_i$ have been defined for some $i < |\E|$ in such a way that the collection $\{ \v(f_k) \}_{k = 1}^{i}$ is linearly independent.  Thus the collection $\{ \v(f_k) \}_{k = 1}^{i}$ does not span Met($\X$), so by Lemma \ref{lemma 4.1} there exists $g \in \text{Simp}(\X, \R^{d})$ such that $\{ \v(f_k) \}_{k = 1}^{i} \cup \{ \v(g) \}$ is also linearly independent.  Let $\{ g_l \}_{l = 1}^{d}$ denote the component functions of $g$.  Since $\v(g) = \sum_{l = 1}^d \v(g_l)$, there must exist some component function $g_j$ such that $\v(g_j)$ is not in Span($\{ \v(f_k) \}_{k = 1}^{i}$).  Choose $f_{i + 1} = g_j$.  Then the collection $\{ \v(f_k) \}_{k = 1}^{i+1}$ is linearly independent.

This method constructs a function $f \in \text{Simp}(\X, \R^{|\E|})$ so that the collection of component functions under $\v$, $\{\v(f_k) \}_{k = 1}^{|\E|}$, is linearly independent, and thus spans Met($\X$).  Therefore $f$ is a simplicial map with a spanning metric.
\end{proof}

\subsection{Proof of Theorem \ref{main thm 2}}

\begin{proof}[Proof of Theorem \ref{main thm 2}]

Let $ f \in \text{Simp}(\X, \R^{2n + 1}) $ be an embedding and let $ h \in \text{Simp}(\X, \R^{|\E|}) $ be a simplicial map with a spanning metric, whose existence is guaranteed by Corollary \ref{corollary 4.1}.  Let $ \{ h_k \}_{k = 1}^{|\E|} $ be the component functions of $h$.  Then by assumption, $ \{ \v(h_k) \}_{k = 1}^{|\E|} $ spans Met($\X$).  So there exists $ \alpha_1, \alpha_2, ... , \alpha_{|\E|} \in \R $ such that
	\begin{equation*}
	g^2 - \v(f) = \sum_{k = 1}^{|\E|} \alpha_k \v(h_k). 
	\end{equation*}
Thus
	\begin{equation*}
	g^2 = \v(f) + \sum_{k = 1}^{|\E|} \alpha_k \v(h_k). 
	\end{equation*}
Let $ p = 2n + 1 + \left| \{ \alpha_k | \alpha_k \geq 0 \} \right| $ and let $ q = | \{ \alpha_k | \alpha_k < 0 \} | $.  Then define $ z \in \text{Simp}(\X, \Rpq) $ by
$$ z = f \bigoplus^{|\E|}_{k = 1 \, , \, \alpha_k \geq 0} \sqrt{\alpha_k} h_k \bigoplus^{|\E|}_{l = 1 \, , \, \alpha_l < 0} \sqrt{|\alpha_l |} h_l $$
 and notice that
	\begin{align*}
	\v(z) &= \v(f) \, + \sum^{|\E|}_{k = 1 \, , \, \alpha_k \geq 0} \alpha_k \v(h_k) \, - \sum^{|\E|}_{l = 1 \, , \, \alpha_l < 0} | \alpha_l | \v(h_l)  \\
	&= \v(f) \, + \sum^{|\E|}_{k = 1 \, , \, \alpha_k \geq 0} \alpha_k \v(h_k) \, + \sum^{|\E|}_{l = 1 \, , \, \alpha_l < 0} \alpha_l \v(h_l) \\
	&= \v(f) + \sum^{|\E|}_{k = 1} \alpha_k \v(h_k) = g^2 .
	\end{align*}

 Therefore, $z$ is a simplicial isometry of $\X$ into $\Rpq$ where $p + q = 2n + 1 + |\E|$ and $p \geq 2n + 1$. $z$ is an embedding since $f$ is.

\end{proof}

\vskip 10pt

\section{Proof of Theorem \ref{main thm 3}}\label{section 5}

For this section let $(\X, \T, g)$ be an $n$-dimensional indefinite metric polyhedron with vertex set $\V$ and edge set $\E$, assume that $d = \text{max}\{ \text{deg}(v) | v \in \V \} < \infty$, and let $q = \text{max} \{ 2n + 1, d \}$.  For a vertex $v$ the closed star of $v$ will be denoted by $St(v)$.  We define $St^2(v) := \bigcup_{u \in St(v)} St(u)$ and for any $k \in \mathbb{N}$ we define $St^{k + 1}(v) := \bigcup_{u \in St^k(v)} St(u)$.

An outline of the proof is as follows.  The first step is to construct, for each $v \in \V$, a compact indefinite metric polyhedron denoted by $(\S_v, \T_v, \g_v)$.  Then the vertex set $\V$ is partitioned into $D = d^3 - d^2 + d + 1$ classes $\{ \mathcal{C}_i \}_{i = 1}^{D}$ which satisfy that if $u, v \in \mathcal{C}_i$ then\footnote{One easily sees that this condition is symmetric} $u \nin St^3(v)$.  Note that this is equivalent to the statement $int(St^2(u)) \cap int(St^2(v)) = \emptyset$ where $int()$ means ``interior".  Now, for each $v$ in a fixed class $\mathcal{C}_i$, a simplicial isometric embedding $\a_v: \S_v \rightarrow \R^{2q, 2q}$ is constructed which satisfies that if $u, v \in \mathcal{C}_i$ then $\a_u (\S_u) \cap \a_v (\S_v) = \vec{0}$.  This allows for the construction of a simplicial map $\b_i: \X \rightarrow \R^{2q,2q}$ (for each $1 \leq i \leq D$).  Then the simplicial isometric embedding is 
	\begin{equation*}
	\l := \bigoplus_{i = 1}^{D} \b_i : \X \longrightarrow \R^{p,p}
	\end{equation*}
where $p = 2qD = 2q(d^3 - d^2 + d + 1)$.

Before proceeding to the proof of Theorem \ref{main thm 3} it should be noted that, in exactly the same way as for Theorem \ref{main thm 1}, there are the following two Corollaries to Theorem \ref{main thm 3}.
	
\begin{cor}\label{corollary 5.1}
Let $(\X, \T, g)$ be an $n$-dimensional indefinite metric polyhedron with vertex set $\V$ and suppose that $d = \text{max} \{ \text{deg}(v) | v \in \V \} < \infty$. Let $\ds{q = \text{max} \{ d, \, 2n \} }$.  Then there exists a simplicial isometric local embedding of $\X$ into $\ds{\R^{p,p}}$ where $p = q(d^3 - d^2 + d + 1)$.
\end{cor}
	
\begin{cor}\label{corollary 5.2}
Let $(\X, \T, g)$ be an $n$-dimensional indefinite metric polyhedron with vertex set $\V$ and suppose that $d = \text{max} \{ \text{deg}(v) | v \in \V \} < \infty$. Then there exists an isometric immersion of $\X$ into $\ds{\R^{p,p}}$ where $p = d(d^3 - d^2 + d + 1)$.
\end{cor}
	
In comparing the above two Corollaries to Theorem \ref{main thm 3} we see that there is a ``2" missing from the dimensional requirements.  During the proof of Theorem \ref{main thm 3} we will indicate where this difference comes from.

\subsection{Construction of the compact indefinite metric polyhedron $(\mathcal{S}_v, \T_v, \g_v)$}

Let $v \in \V$.  The Polyhedron $\S_v$ will look like the cone of $St(v)$, and in fact that would work.  But in an attempt to keep the dimension of the embedding space as small as possible the construction is altered some as follows.

The polyhedron $\S_v$ and the triangulation $\T_v$ are constructed at the same time.  Begin the construction of $\S_v$ with the entire complex $St(v)$.  Then adjoin a vertex denoted by $v^*$ as follows.  Glue in an edge between $v^*$ and a vertex $u$ on the boundary of $St(v)$ if and only if there exists an edge in $\T$ which is adjacent to $u$ and \emph{not} contained in $St(v)$. Do \emph{not} connect an edge between $v$ and $v^*$.  Then for any $2 \leq k \leq n$, if there exist $k$ vertices on the boundary of $St(v)$ which are contained in the boundary of a $k$ simplex in $\X \setminus int(St(v))$, glue in a $k$-dimensional simplex using those $k$ vertices and $v^*$ (see Figure 1).

This completes the construction of the polyhedron $(\S_v, \T_v)$.  It is clear that $\S_v$ is compact.  It is important to note that $\S_v$ has dimension less than or equal to $n$ and the maximal degree of any vertex is less than or equal to $d$, so it meets the criteria of Theorem \ref{main thm 1}.  What is left to do is to describe the indefinite metric $\g_v$.

Let $e$ be an edge of $\S_v$.  Either $e$ is adjacent to $v$ or $e$ is not adjacent to $v$.  In the latter case, simply define $\g_v(e) := 0$.  So, in particular, notice that every edge adjacent to $v^*$ has intrinsic length $0$.  In the former case, when $St(v)$ is considered as a subcomplex of $\X$, the edge $e$ has an intrinsic length $g(e)$.  Define $\g_v(e) := \frac{1}{\sqrt{2}} g(e)$ (see Figure \ref{firstfig}).

\subsection{Partitioning $\V$ into $D = d^3 - d^2 + d + 1$ classes and isometric embeddings of $\mathcal{S}_v$ into $\R^{2q,2q}$}

The first goal here is to partition $\V$ into $D = d^3 - d^2 + d + 1$ classes $\{ \mathcal{C}_i \}_{i = 1}^{D}$ which satisfy that, if $u, v \in \mathcal{C}_i$, then $u \nin St^3(v)$.  Since $\T$ is locally finite, $\V$ is countable.  So enumerate $\V$ in some way.  The class in which a given vertex resides in is defined recursively.  Put the first vertex in $\mathcal{C}_1$.  Then assume that the classes of all of the vertices before a given vertex $v$ have been determined.  This vertex $v$ is connected by an edge to at most $d$ other vertices of $\T$.  Each of these vertices is connected by an edge to at most $d - 1$ vertices of $\T$ other than $v$ (the $d - 1$ is since they are all connected to $v$).  Thus, $\partial St^2(v)$ contains at most $d(d - 1)$ vertices and $St^2(v)$ contains at most $d(d - 1) + d = d^2$ vertices other than $v$.  But, by the same logic as before, each vertex in the boundary of $St^2(v)$ is connected by an edge to at most $d - 1$ vertices not in $St^2(v)$.  So $St^3(v)$ contains at most $d(d - 1)^2 + d^2 = d^3 - d^2 + d$ vertices other than $v$.  Then since there are $d^3 - d^2 + d + 1$ classes, there always exists a class $\mathcal{C}_i$ such that $St^3(v)$ does not contain any of the vertices already in $\mathcal{C}_i$.  Place $v$ into such a class, completing the definition of the classes $\{ \mathcal{C}_i \}_{i = 1}^{D}$.

For the following discussion let us fix a class $\mathcal{C}_i$.  By Theorem \ref{main thm 1}, for each $v \in \mathcal{C}_i$ there exists a simplicial isometric embedding $h_v: \S_v \rightarrow \R^{q,q}$.  By composing with a translation it may be assumed that $h_v(v^*) = \vec{0}$.  In what follows we construct, for each $v \in \mathcal{C}_i$, a (linear) isometric embedding $\iota_v: \R^{q,q} \rightarrow \R^{2q,2q}$ (see Figure \ref{secondfig}).  Then $\a_v := \iota_v \circ h_v$ will be the desired simplicial isometric embedding of $\S_v$ into $\R^{2q,2q}$.

\begin{figure}[tb]
\begin{center}
\begin{tikzpicture}[scale=0.8]

\draw (3,2.5)node{$\X$ ($St(v)$ bolded, gray denotes 2-simplex)};
\draw[fill=gray!40] (0,2) -- (3,2) -- (2,1) -- (0,2);
\draw[fill=gray!40] (3,2) -- (4,0) -- (2,1) -- (3,2);
\draw (4,0) -- (2,-1);
\draw[fill=gray!40] (0,0)node[left]{$v$} -- (0,2) -- (2,1) -- (0,0);
\draw[line width=0.5mm] (0,0) -- (0,2) -- (2,1);
\draw (2,1)-- (0,0);
\draw[line width=0.5mm] (2,-1)-- (2,1);
\draw[fill=gray!40] (0,0) -- (2,-1) -- (0,-2)-- (0,0);
\draw (0,0) -- (2,-1);
\draw[line width=0.5mm] (2,-1) -- (0,-2)-- (0,0);
\draw (2,-1) -- (4,-1);
\draw (4,0) -- (4,-1);
\draw (0,1)node[left]{2};
\draw (0,-1)node[left]{3};
\draw (1,-1.5)node[below]{-3};
\draw (0.75,-0.5)node[below]{$\sqrt{2}$};
\draw (1,0.5)node[below]{1};
\draw (1.5,2)node[below]{-9};
\draw (1,1.5)node[below]{-1};
\draw (2.5,1.5)node[right]{0};
\draw (3.5,1)node[right]{-4};
\draw (2,0)node[right]{11};
\draw (3,0.5)node[above]{-1};
\draw (3,-0.5)node[above]{7};
\draw (3,-1)node[below]{-1};
\draw (4,-0.5)node[right]{100};

\draw (5,0) -- (7,0);
\draw (6.9,0.1) -- (7,0);
\draw (6.9,-0.1) -- (7,0);

\draw (11, 2.5)node{$S_v$};
\draw[fill=gray!40] (8,0)node[left]{$v$} -- (8,2) -- (10,1) -- (8,0);
\draw[fill=gray!40] (8,2) -- (14,0)node[right]{$v^*$} -- (10,1) -- (8,2);
\draw[fill=gray!40] (8,0) -- (10,-1) -- (8,-2) -- (8,0);
\draw[line width=0.5mm] (10,1) -- (10,-1);
\draw (10,-1) -- (14,0);
\draw[line width=0.5mm] (8,0) -- (8,2) -- (10,1);
\draw[line width=0.5mm] (10,-1) -- (8,-2) -- (8,0);
\draw (8,1)node[left]{$\sqrt{2}$};
\draw (8,-1)node[left]{$\frac{3}{\sqrt{2}}$};
\draw (9,0.5)node[below]{$\frac{1}{\sqrt{2}}$};
\draw (9,-0.5)node[below]{1};
\draw (11,1)node[above]{0};
\draw (9,1.5)node[below]{0};
\draw (10,0)node[right]{0};
\draw (12,0.5)node[below]{0};
\draw (12,-0.5)node[below]{0};
\draw (9,-1.5)node[below]{0};

\end{tikzpicture}
\end{center}
\caption{The assignments of $g$ and $\g_v$, respectively, are denoted along each edge.}
\label{firstfig}
\end{figure}
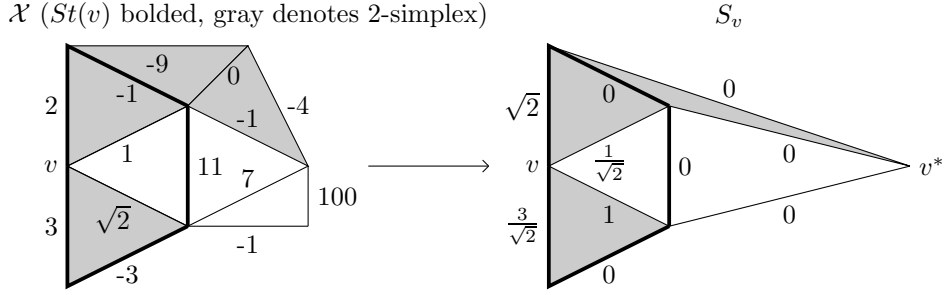

Since $\V$ is countable, $\mathcal{C}_i$ is countable.  So there exists an injection $\mu:\mathcal{C}_i \rightarrow \mathbb{N}$.  Thinking of $\R^{2q,2q}$ as $\R^{q,q} \times \R^{q,q}$, define $\iota_v: \R^{q,q} \rightarrow \R^{2q,2q}$ by 
	\begin{equation*}
	\iota_v(\vec{x}) = \left( \sqrt{\frac{1}{\mu(v)}} \, \vec{x} , \sqrt{1 - \frac{1}{\mu(v)}} \, \vec{x} \right). 
	\end{equation*}
$\iota_v$ is a linear isometry, as desired.  Note that since $h_v(v^*) = \vec{0}$, $\a_v(v^*) = \iota_v(\vec{0}) = \vec{0}$.  Also notice that for $v, w \in \mathcal{C}_i$, $\a_v(\S_v) \cap \a_w(S_w) = \vec{0}$.

It should be noted that the purpose of the map $\iota_v$ is to ensure that this construction leads to an embedding.  It is not necessary if the only requirement is that the proof leads to either a local embedding or an immersion.  This is why Corollaries \ref{corollary 5.1} and \ref{corollary 5.2} do not require the ``2" that is present in Theorem \ref{main thm 3}.

\subsection{Wrapping up the construction}

What is left of the proof is to use the $\a_v's$ of a class $\mathcal{C}_i$ to construct the simplicial map $\b_i:\X \rightarrow \R^{2q,2q}$, and then showing that $\l := \bigoplus_{i = 1}^{D} \b_i : \X \rightarrow \R^{p,p}$ where $p = 2q(d^3 - d^2 + d + 1)$ is an isometric embedding.  

Since $\b_i$ is to be simplicial, it need only be defined on the vertices of $\T$.  So let $u \in \V$ be arbitrary.  If $u \in St(v)$ for some $v \in \mathcal{C}_i$ then define $\b_i(u) := \a_v(u)$ (where, for $\a_v(u)$ to make sense, $St(v)$ is considered as a subcomplex of $\S_v$).  Otherwise, $u \nin St(v)$ for all vertices $v \in \mathcal{C}_i$.  In this case, define $\b_i(u) := \vec{0}$.  This is a well-defined construction since the closed stars of vertices in $\mathcal{C}_i$ are disjoint.

\subsection*{Showing $\l$ is an isometry}

In order to show that $\l$ is an isometry, $\v(\b_i)(e)$ needs to be analyzed (for each $i$ and on each edge $e \in \E$).  

So let $e \in \E$ be arbitrary, and let $u$ and $v$ denote the vertices adjacent to $e$, respectively.  This is broken down into four cases.  The first and most important case is when one of the vertices $u$ or $v$ is in $\mathcal{C}_i$.  Without loss of generality assume that it is $v$.  Then $u, v \in St(v)$ and thus $\b_i(v) = \a_v(v)$ and $\b_i(u) = \a_v(u)$.  So $\v(\b_i)(e) = \v(\a_v)(e) = s(\g_v(e)) = \frac{1}{2} s(g(e)) = \frac{1}{2} g^2(e)$.  

For the last three cases, assume that neither $u$ nor $v$ is in $\mathcal{C}_i$.  Case 2 is when there exists $w \in \mathcal{C}_i$ such that both $u, v \in St(w)$, or equivalently $e \subseteq \partial St(w)$.  This case is analogous to the above but this time, due to the definition of $\g_w$, $\v(\b_i)(e) = s(\g_w(e)) = 0$.  For Case 3, assume that there exists $w \in \mathcal{C}_i$ such that exactly one of $u$ or $v$ is in $St(w)$, say $u \in St(w)$.  It is important to note that there cannot exist $x \in \mathcal{C}_i$ such that $v \in St(x)$, for otherwise it would occur that $x \in St^3(w)$ which violates the construction of the class $\mathcal{C}_i$.  So here it is seen that $\b_i(u) = \a_w(u)$ and $\b_i(v) = \vec{0} = \a_w(w^*)$.  Therefore $\v(\b_i)(e) = \v(\a_w)(e)$\footnote{where we consider $e$ as the edge between $u$ and $w*$ in $\S_w$}$ = 0$.  The last case is when neither $u$ nor $v$ is in the closed star of any member of $\mathcal{C}_i$.  But in this case both vertices are mapped to $\vec{0}$ and hence $\v(\b_i)(e) = 0$.  

The key point to note here is that the only edges $e \in \E$ for which $\v(\b_i)(e) \neq 0$ are those which are adjacent to a member of $\mathcal{C}_i$.  And in this case $\v(\b_i)(e) = \frac{1}{2} g^2(e)$.  But since each edge is adjacent to exactly two vertices (both of which are in different classes), and since $\v$ is additive with respect to concatenation of maps, for every edge $e \in \E$ it is the case that $\v(\l)(e) = \sum_{i = 1}^{D} \v(\b_i)(e) = \frac{1}{2} g^2(e) + \frac{1}{2} g^2(e) = g^2(e)$.  Hence $\l$ is an isometry.

\begin{figure}
\begin{center}
\begin{tikzpicture}[scale=0.8]

\draw (-2.2,0) -- (2.2,0);
\draw (2.1, 0.1) -- (2.2,0);
\draw (2.1, -0.1) -- (2.2,0);
\draw (-2.1, 0.1) -- (-2.2,0);
\draw (-2.1, -0.1) -- (-2.2,0);
\draw (2.2,0)node[above]{$\R^q_q$};

\draw (3.5,0) -- (5,0);
\draw (4.9,0.1) -- (5,0);
\draw (4.9,-0.1) -- (5,0);
\draw (4.25,0)node[above]{$\iota_v$};

\draw[line width=0.5mm] (5.5,0) -- (10.5,0);
\draw (8,-2.5) -- (8,2.5);
\draw[line width=0.5mm] (6.23223,-1.76777) -- (9.76777,1.76777);
\draw[line width=0.5mm] (6.55662,-2.04124) -- (9.44338,2.04124);
\draw (10.5,0)node[above]{when $\mu(v) = 1$};
\draw (9.76777,1.76777)node[right]{when $\mu(v) = 2$};
\draw (9.44338,2.3)node[right]{when $\mu(v) = 3$};
\draw (9.25,2.16506)node{.};
\draw (9.11803,2.23607)node{.};
\draw (9.02062,2.28218)node{.};
\draw (8.94491,2.31455)node{.};
\draw (10.5,0)node[below]{$\R^q_q$};
\draw (8,2.5)node[left]{$\R^q_q$};

\end{tikzpicture}
\end{center}
\caption{The images of $\iota_v$ for various values of $\mu(v)$.}
\label{secondfig}
\end{figure}
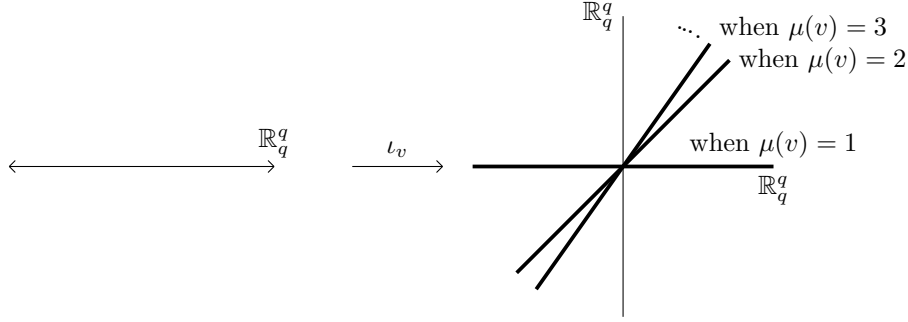

\subsection*{Showing $\l$ is an embedding}

Let $x, y \in \X$ with $x \neq y$.  Let $v \in \V$ be such that $x \in int(St(v))$ and let $i$ be the index such that $v \in \mathcal{C}_i$.  Note that since $x$ is in the interior of $St(v)$, $\b_i(x) = \a_v(x) \neq \vec{0}$.  What will be shown is that $\b_i(x) \neq \b_i(y)$, and therefore $\l(x) \neq \l(y)$.  

Clearly, if $y \in St(v)$, then $\a_v(x) \neq \a_v(y) \Longrightarrow \b_i(x) \neq \b_i(y) \Longrightarrow \l(x) \neq \l(y)$.  So suppose $y \nin St(v)$.  If $y \in St^2(w)$ for any $w \in \mathcal{C}_i$ with $w \neq v$, then $\b_i(y) = \a_w(y)$.  But $\a_v(\S_v) \cap \a_w(\S_w) = \vec{0}$ and $\b_i(x) \neq \vec{0}$, and thus $\b_i(x) \neq \b_i(y)$.  If $y \nin St^2(w)$ for any $w \in \mathcal{C}_i$ (including $v$) then $\b_i(y) = \vec{0}$ and therefore $\b_i(y) \neq \b_i(x)$.

So the only case left is when $y \in St^2(v) \setminus St(v)$.  Define a simplicial map\footnote{The purpose for gluing extra simplices onto $St(v)$ in the construction of $\S_v$ was so that we can extend this map simplicially over all of $\X$} $\pi_v: \X \rightarrow \S_v$ by mapping each vertex of $St(v)$ to itself and mapping every other vertex of $\T$ to $v^*$.  Note that $\pi_v$ maps all of $\X \setminus St^2(v)$ to $v^*$.  So for all $v \in \V$ there is the following sequence of maps:

\begin{center}
\begin{tikzpicture}
\matrix(m)[matrix of math nodes,
row sep=2.6em, column sep=4em,
text height=1.5ex, text depth=0.25ex]
{\X & \S_v & \R^{q,q} & \R^{2q,2q} \\
};
\path[->,font=\scriptsize,>=angle 90]
(m-1-1) edge node[above] {$\pi_v$} (m-1-2)
(m-1-2) edge node[above] {$h_v$} (m-1-3)
(m-1-3) edge node[above] {$\iota_v$} (m-1-4);
\end{tikzpicture}
\end{center}

\n When restricted to $St^2(v)$ it is easy to see that $\b_i = \iota_v \circ h_v \circ \pi_v$.  But since $y \in St^2(v) \setminus St(v)$, $\pi_v(x) \neq \pi_v(y)$.  Then, since $h_v$ and $\iota_v$ are embeddings, $\b_i(x) \neq \b_i(y)$.  Hence $\l$ is an embedding and therefore the proof of Theorem \ref{main thm 3} is complete.

\vskip 10pt

\section{Equivalence of the definitions of an indefinite metric polyhedron}\label{section 7}

\subsection{Assignment of the quadratic form}

Let $x \in \X$ be a point.  Then there is a unique $k$-dimensional simplex $\s_x = \langle v_0, v_1, ..., v_k \rangle \in \T$ such that $x$ is interior to $\s_x$.  So one can consider a $k$-dimensional tangent space at $x$, denoted by $T_x \X$, whose dimension certainly depends on the triangulation $\T$.  Under the simplicial isometric embeddings produced in sections 3, 4 and 5, $T_x \X$ can be considered as a $k$-dimensional affine subspace of $\Rpq \cong \R^N$ (where $N = p + q$).  Consider the collection of vectors 
	\begin{equation*}
	\mathcal{B}_x = \{ v_1 - v_0, v_2 - v_0, ..., v_k - v_0 \} 
	\end{equation*}
where the difference makes sense since the complex is being considered as a subspace of $\R^N$.  Clearly $\Bx$ is a basis for $T_x$.  So associate to $\s_x$ the $k \times k$ symmetric matrix $G(\s_x)$ defined by:
	\begin{equation*}
	G(\s_x)_{ij} = (\langle (v_i - v_0) , (v_j - v_0) \rangle )_{ij} 
	\end{equation*}
where the inner product is take in $\Rpq$.  $G(\s_x)$ is simply the \emph{Gram matrix} associated with $\mathcal{B}_x$.  At first glance it seems that this definition might depend on the isometric embedding of $\s_x$.  What is done now is to (quickly) show that this is not true.  The key is that:
	\begin{align*}
	g^2(\eij) &= \langle (v_i - v_j) , (v_i - v_j) \rangle   \\
	&= \langle ((v_i - v_0) - (v_j - v_0)) , ((v_i - v_0) - (v_j - v_0)) \rangle   \\
	&= g^2(e_{0i}) -2 \langle (v_i - v_0) , (v_j - v_0) \rangle + g^2(e_{0j}) 
	\end{align*}
and so 
	\begin{equation*}
	\langle (v_i - v_0) , (v_j - v_0) \rangle \, = \frac{1}{2} \left( g^2(e_{0i}) + g^2(e_{0j}) - g^2(\eij) \right) 
	\end{equation*}
\n where the edge notation is the same as always.

This shows that the matrix $G(\s_x)$ depends only on the intrinsic (indefinite) metric $g$.  Of course, $G(\s_x)$ also depends on the ordering of the vertices of $\s_x$.  But changing the order of the vertices of $\s_x$ just changes the coordinates of $\Bx$.  Thus, $G(\s_x)$ is well-defined when considered as a symmetric bilinear form on $T_x \X$ (or equivalently on the tangent space to any point interior to $\s_x$).

In this way one can associate a symmetric bilinear form to every simplex of $\T$.  This form allows us to assign an \emph{energy} to any straignt line segment interior to any closed simplex.  For if $\s \in \T$ is a $k$-dimensional simplex and $a, b \in \s$ with barycentric coordinates $(\a_i)_{i = 0}^{k}$ and $(\b_i)_{i = 0}^{k}$ respectively, then the energy of the straight line segment (in $\s$) from $a$ to $b$ is 
	\begin{equation*}
	v^T G(\s) v 
	\end{equation*}
\n where $v \in \R^k$ is defined as
	\begin{equation*}
	v = (\a_i - \b_i)_{i = 1}^{k}. 
	\end{equation*}

It is easy to see that the energy of a line segment is well-defined at the intersection of any collection of simplices.  It is also easy to see that the energy assigned to any edge $\eij$ under this definition is $g^2(\eij)$.  Thus, the collection of indefinite metrics on $(\X, \T)$ is in one-to-one correspondence with assignments of a symmetric bilinear form to each simplex of $\T$ that agree (meaning they assign the same energy to any line segment) on the intersection of any two simplices.

\subsection{Euclidean and Minkowski polyhedra}

Let $(\X, \T, g)$ be an indefinite metric polyhedron and let $G$ be the symmetric bilinear form defined as above with respect to $g$.  $\X$ is a \emph{Euclidean Polyhedron} if $G(\s)$ is positive definite for all $\s \in \T$.  $\X$ is a \emph{Minkowski Polyhedron} if $G(\s)$ is non-degenerate for all $\s \in \T$

A $k$-dimensional simplex $\s \in \T$ admits a simplicial isometric embedding into Euclidean space of dimension $k$ if and only if $G(\s)$ is positive definite.  For a proof see \cite{Bhatia}.  If a $k$-dimensional simplex admits a simplicial isometric embedding into $\Rpq$ with $p + q = k$, then the signature of $G(\s)$ will not contain any zeroes since the inner product on $\Rpq$ is non-degenerate.  This justifies the above definition.

Note that for a general indefinite metric polyhedron, the quadratic form $G(\s)$ can have zeroes in its signature.  Theorems \ref{main thm 1}, \ref{main thm 2} and \ref{main thm 3} do not contradict the above statement since $p + q > n$ in these Theorems.

\section{An example verifying the sharpness of the dimension requirements for Theorem \ref{main thm 1} and the assumptions in Theorem \ref{main thm 3}}\label{section 6}

First we consider the case when $n = 1$, and then this case will be used to illustrate an example which works for arbitrary $n$.  If the dimension requirements for Theorem \ref{main thm 1} did \emph{not} depend on $d := \text{max} \{ \text{deg}(v) | v \in \V \}$, then the Theorem would state that every 1-dimensional compact indefinite metric polyhedron admits a simplicial isometric embedding into $\R^{3,3}$.  Let $(\X, \T, g)$ be the 1-skeleton of a 4-simplex and define $g(e) = 1$ for every edge $e \in \T$.  Clearly $\X$ admits a simplicial isometric embedding into $\mathbb{E}^4$, but not into $\mathbb{E}^3$.  The claim is that $\X$ \emph{does not} admit a simplicial isometric embedding into $\R^{3,3}$.  To see this, just note that the Gram matrix associated to any such embedding will have signature $(4,0)$.  So if such a simplicial isometric embedding existed, then the restriction of the global inner product in $\R^{3,3}$ to any subspace containing the image of $\X$ must have at least 4 positive eigenvalues.  But any such restriction can have no more than 3 positive eigenvalues.  The same example, but with edge lengths of -1, shows that $d$ is necessary in the negative signature as well.

This same construction, but using the $n$-skeleton of an $N$-simplex with $N \geq 2n + 2$, works for arbitrary $n$.  It will be seen next section that this example\footnote{Actually, this fact holds for all $n$-dimensional indefinite metric polyhedra, not just this example.} \emph{does} admit a pl isometric embedding into $\R^{2n,n} \subseteq \R^{2n+1, 2n+1}$, so the necessity of $d$ in the dimensionality of the target Minkowski space vanishes when one only requires that the isometric embedding be pl instead of simplicial!  

With this example in mind it is easy to see that the assumption $d < \infty$ is necessary in Theorem \ref{main thm 3}.  Let $(\X_N, \T_N, g_N)$, $N \geq 2n + 2$, denote the above example.  Then construct an indefinite metric polyhedron $(\X, \T, g)$ by gluing together all of the previously mentioned polyhedra at a common vertex.  Since the dimensionality of the target Minkowski space must increase without bound as $N \rightarrow \infty$, $(\X, \T, g)$ does not admit a simplicial isometric embedding into $\Rpq$ for any $p, q$.

\end{document}